\documentclass[12pt]{amsart}
\include{xy}
\xyoption{all}
\newcommand{\pad}{\operatorname{pad}}
\newcommand{\minrank}{\operatorname{mrank}}

\usepackage{fullpage}
\newtheorem{theorem}{Theorem}
 
\newtheorem{lemma}[theorem]{Lemma}

\theoremstyle{definition}

\newtheorem{remark}[theorem]{Remark}
\newtheorem{example}[theorem]{Example}
\newcommand{\id}{\operatorname{id}}
\newcommand{\rank}{\operatorname{rank}}

\newcommand{\R}{{\mathbb R}}
\newcommand{\C}{{\mathbb C}}

\newcommand{\F}{{\mathbb F}}
\newcommand{\Q}{{\mathbb Q}}

\newcommand{\Mat}{\operatorname{Mat}}
\title{On the equivalence between low rank matrix completion and tensor rank}
\author{Harm Derksen}
\thanks{The author was partially supported by NSF grant DMS 1302032.}
\begin{document}
\maketitle
\begin{abstract}
The Rank Minimization Problem asks to find a matrix of lowest rank inside a linear variety of the space of $n\times m$ matrices. The  Low Rank Matrix Completion problem asks to complete a partially filled matrix such that the resulting matrix has smallest possible rank. The Tensor Rank Problem asks to determine the rank of a tensor. 
We show that these three problems are equivalent: each one of the problems can be reduced to the other two.
\end{abstract}

\section{Introduction}
Suppose that $\F$ is a field. We will consider the following computational problems:
\subsection*{Rank Minimization (RM)}
Given matrices $A,B_1,\dots,B_s\in \Mat_{n,m}(\F)$, find $x_1,\dots,x_s\in \F$ for which
$\rank(A+x_1B_1+\cdots+x_sB_s)$ is minimal.
\subsection*{Special Rank Minimization (1--RM)}
This problem is the same as the Rank Minimization Problem, but we assume that the matrices $B_1,\dots,B_s$ have rank $1$.

\subsection*{Low Rank Matrix Completion (LRMC)} Given an $n\times m$ matrix $A$ which is partially filled with entries of $\F$,
fill in the remaining entries such that the resulting matrix has minimal rank.

\subsection*{Tensor Rank (TR)} Given  a tensor $T\in \F^{n_1}\otimes \F^{n_2}\otimes \cdots \otimes \F^{n_d}$
Find pure tensors $v_j=v_j^{(1)}\otimes \cdots \otimes v_j^{(d)}$ for $j=1,2,\dots,r$ such that $T=v_1+\cdots+v_r$
and $r$ is minimal. (This minimal $r$ is called the {\em tensor rank} of $T$, see ~\cite{Hitchcock}. )
\subsection*{Order 3 Tensor Rank (3--TR)} This is the Tensor Rank problem where the order $d$ is equal to $3$.\\

%If the field $\F$ is the field $\R$ of real numbers, then we also consider:

%\subsection*{Positive Definite low rank Matrix Completion (PSMC)} Given an $n\times n$ partially filled symmetric matrix $A$ with entries in $\R$, fill in the remaining entries such that the resulting matrix is symmetric, positive semidefinite of the lowest possible rank.

The LRMC problem has many applications. One such application is {\em collaborative filtering}. A typical example is the {\em  Netflix problem}. Suppose that the rows of a matrix correspond to users,
and the columns correspond to movies. Each user has only rated   some of the movies. This leads to a partially filled matrix. To predict whether users would like a certain movie that they have not seen yet, one would like to complete the matrix. Because preferences only depend on a few parameters, it is reasonable to assume that the rank of the completed matrix should be small. The LRMC problem is known to be NP complete (see~\cite[Theorem 3.1]{Peeters}). For the field $\F=\R$, some algorithms have been proposed for the LRMC problem using convex relaxation, for example  \cite{F,CR,CCS,CR}. 

If the matrices $B_1,\dots,B_s$ in the Rank Minimization problems are matrices that have only one nonzero entry, 
then the RM problem becomes an LRMC problem. So the RM problem is more general. Over finite fields, RM has applications to coding theory (\cite{BYBJK, TBD}).
For the real numbers, applications are in control (\cite{FHB,FHB2,GG}), systems identification (\cite{LV}) and Euclidean embedding problems (\cite{LLR,Laurent0}).
Some  algorithms have appeared in \cite{FHB,RFP,LV,CCS,MGC}.

The TR problem is NP-complete if the field  $\F$ is  finite  (\cite{Hastad,Hastad2}) and
NP-hard if the field is $\Q$ (\cite{Hastad,Hastad2}), or if it contains $\Q$ (\cite{HL}). Approximation
of a tensor by another tensors of low rank is known as the PARAFAC (\cite{Harshman}) or CANDECOMP (\cite{CC}) model. There are various hearistic approaches  for finding low-rank approximations. See~\cite{KB,TB} for a discussion.
The problems of finding the rank of a tensor, and to approximate tensors with tensors of low rank
 has many applications, such as the complexity of matrix multiplication, fluorescence spectroscopy,  statistics, psychometrics, geophysics and magnetic resonance imaging. 
In general, the rank of higher tensors more ill-behaved than the rank of a matrix (see~\cite{DSL}).

 We will show that these 5 problems are equivalent to each other by giving explicit reductions. Some reductions are obvious: 1--RM is a special case of RM, 3--TR is a special case of TR and LRMC is a special case of 1--RM. The reduction from LRMC to $3$--TR was done in \cite{Derksen}.
 Other reductions are explicitly described in this paper and these reductions are denoted by double arrows. By composition of reduction steps, we can reduce each of the 5 problems to any other problem. In \cite{BBCM}  the authors give an algorithm
 for tensor decompositions and more generally, multi-homogeneous tensors (see also~\cite{BCMT} for symmetric tensor decompositions).  
 A crucial step (Algorithm 3.1, step (1)) in their algorithm asks 
 for a low rank completion of a Hankel matrix (Theorem 3.3.), so their algorithm can also be  viewed as a reduction from TR to RM.

$$
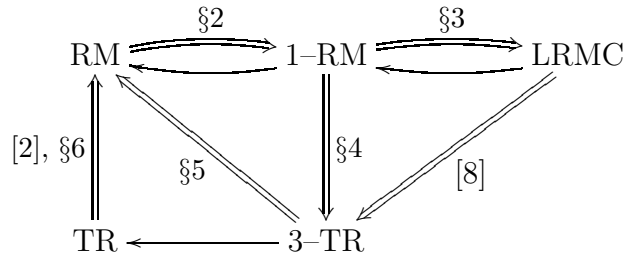
\xymatrix{
& \mbox{RM}\ar@/^/@{=>}[rr]^{\mbox{\small \S\ref{RMto1RM}}} & & \mbox{1--RM}\ar@/^/@{=>}[rr]^{\mbox{\small \S\ref{1RMtoLRMC}}}\ar@{=>}[dd]^{\mbox{\small \S\ref{1RMto3TR}}}\ar@/^/[ll] &&
\mbox{LRMC}\ar@/^/[ll]\ar@{=>}[lldd]^{\mbox{\cite{Derksen}}}\\ 
& &&&&\\
&\mbox{TR}\ar@{=>}[uu]^{\mbox{\small \cite{BBCM},  \S\ref{TRtoRM}}}&& \mbox{3--TR}\ar[ll]\ar@{=>}[lluu]^{\mbox{\small \S\ref{3TRtoRM}}} &&}
$$

\section{Reduction from {\bf RM} to {\bf 1--RM}}\label{RMto1RM}
We first reduce the Rank Minimization Problem with $s=1$ to the Special Rank Minimization Problem.
Suppose that $A,B\in \Mat_{n,m}(\F)$ and we would like to minimize
$\rank(A+xB)$ for $x\in \F$. Let $r=\rank B$. We can write
$$
B=\sum_{i=1}^r v_iw_i^t
$$
where $v_1,\dots,v_r\in \F^n$ and $w_1,\dots,w_r\in \F^m$.
We introduce new variables:
$$
\begin{array}{ll}
x_{i}, & 1\leq i\leq r\\
y_{i,j}, & 1\leq i,j\leq r\mbox{ and $i\neq j$}\\
z_{i,j}, & 1\leq i<j\leq r.
\end{array}
$$
If $i>j$, then we define $z_{i,j}=z_{j,i}$. We will not have such a convention for the $y$'s, so $y_{i,j}$ and $y_{j,i}$ are distinct variables.
Define
\begin{eqnarray*}
{\bf x} & = & (x_1,\dots,x_{r})\\
{\bf y} & = & (y_{i,j}\mid i\neq j, 1\leq i,j\leq r)\\
{\bf z} & = & (z_{i,j}\mid 1\leq i<j\leq r).
\end{eqnarray*}

For $1\leq i<j\leq r$ define the $n\times 2$ matrices
$$
C_{i,j}({\bf x},{\bf y}):=\begin{pmatrix} (y_{i,j}-x_{i})v_{i} & (y_{j,i}-x_{j})v_j
\end{pmatrix}
$$
and the $1\times 2$ matrices
$$
D_{i,j}({\bf y},{\bf z}):=\begin{pmatrix} y_{i,j}-z_{i,j} & y_{j,i}-z_{j,i} \end{pmatrix}.
$$
Also, define an $n\times m$ matrix
$$
E({\bf x}):=A+\sum_{i=1}^r x_i v_iw_i^t.
$$
Let $J$ be the matrix
$$
J({\bf x},{\bf y},{\bf z}):=\left(\begin{array}{c|cccc}
E({\bf x}) & C_{1,2}({\bf x},{\bf y})& C_{1,3}({\bf x},{\bf y}) & \cdots & C_{r-1,r}({\bf x},{\bf y})\\ \hline
0 & D_{1,2}({\bf y},{\bf z}) & 0 &\cdots  &0\\
0 & 0 & D_{1,3}({\bf y},{\bf z}) & & 0\\
\vdots  & \vdots & & \ddots & \vdots\\
0 &  0 & 0 & \cdots & D_{r-1,r}({\bf y},{\bf z})
\end{array}\right)
$$
The size of $J$ is 
$(n+\textstyle\frac{1}{2}r(r-1))\times (m+r(r-1))$.

We will use the convention  $C_{i,j}:=C_{j,i}$ and $D_{i,j}:=D_{j,i}$ if $i>j$.

\begin{theorem}
For every ${\bf x}=(x_1,\dots,x_r)\in \F^r$, ${\bf y}\in \F^{r(r-1)},{\bf z}\in \F^{r(r-1)/2}$ one can choose $x\in \{x_1,\dots,x_r\}$ such that 
$$
\rank J({\bf x},{\bf y},{\bf z})\geq \rank(A+xB).
$$
In particular,
$$
\min_{{\bf x},{\bf y},{\bf z}} \rank J({\bf x},{\bf y},{\bf z})=\min_{x}\rank(A+xB).
$$

\end{theorem}
\begin{proof}
Let 
$$
Z:=\{i\mid \mbox{ $y_{i,j}\neq z_{i,j}$ or $x_i\neq y_{i,j}$ for some $j$}\}.
$$
Choose a $j\in \{1,\dots,r\}$ such that either $j\not\in Z$ or $Z=\{1,\dots,r\}$.
Define $x:=x_j$. Suppose that $x_i\neq x$. If $Z=\{1,\dots,r\}$ then we have $i\in Z$.
Otherwise, $j\not \in Z$ and $x_i\ne x_j=y_{j,i}=z_{j,i}=z_{i,j}$.
So $x_i\neq y_{i,j}$ or $y_{i,j}\neq z_{i,j}$ and it follows that $i\in Z$. We conclude that $x_i\neq x$ implies $i\in Z$.

If $y_{i,j}\neq x_i$, then $C_{i,j}$  contains column that is nonzero multiple $v_i$.
If $y_{i,j}\neq z_{i,j}$ then we wipe $C_{i,j}$ with $D_{i,j}$ using  elementary row operations
to obtain a matrix $C_{i,j}'$ which has a column that is a nonzero multiple of $v_i$.
So after some elementary row operations, we obtain a matrix
$$
J'=\left(\begin{array}{c|cccc}
E & C_{1,2}'& C_{1,3}' & \cdots & C_{r-1,r}'\\ \hline
0 & D_{1,2} & 0 &\cdots  &0\\
0 & 0 & D_{1,3} & & 0\\
\vdots  & \vdots & & \ddots & \vdots\\
0 &  0 & 0 & \cdots & D_{r-1,r}
\end{array}\right)=\left(\begin{array}{c|c}E & C'\\ \hline  0 & D\end{array}\right)
$$
such that for every $i\in Z$, $C'$ has a column that is a nonzero multiple of the column vector  $v_i$.

Using elementary column operations, we can replace the submatrix
$$
E({\bf x})=E(x_1,\dots,x_r)=A+\sum_{i=1}^r x_iv_iw_i^t
$$
of $J'$ by
$$
E(x,x,\dots,x)=A+x\sum_{i=1}^r v_iw_i^t=A+xB=E({\bf x})+\sum_{i\in Z}(x-x_i)v_iw_i^t.
$$
So after elementary row and column operations, $A+xB$ becomes a submatrix of $J({\bf x},{\bf y},{\bf z})$.
This proves that
$$
\rank J({\bf x},{\bf y},{\bf z})\geq \rank(A+xB).
$$
If we set $x_i=y_{i,j}=z_{i,j}=x$ for all $i,j$, then we get
$$
J=\begin{pmatrix} A+xB & 0\\ 0 & 0\end{pmatrix}.
$$
This proves that
$$
\min_{{\bf x},{\bf y},{\bf z}} \rank J({\bf x},{\bf y},{\bf z})=\min_{x}\rank(A+xB).
$$
\end{proof}
We can write $J$ in the form
$$
J({\bf x},{\bf y},{\bf z})=\begin{pmatrix}
A & 0\\  0 & 0\end{pmatrix}+\sum_{i=1}^r x_i F_i+\sum_{i\neq j} y_{i,j}G_{i,j}+\sum_{i<j} z_{i,j}H_{i,j}.
$$
The crucial point here is that $F_i,G_{i,j},H_{i,j}$ all have rank 1. In fact, $G_{i,j}$ has only one nonzero column,
and $H_{i,j}$  has only one nonzero row.
\begin{example}\label{ex:2}
Suppose that 
$$A=\begin{pmatrix} 0 & 1\\ -1 & 0\end{pmatrix},\quad B=\begin{pmatrix} 1 & 0 \\ 0 & 1\end{pmatrix}.
$$
Let
$$
v_1=w_1=\begin{pmatrix} 1 \\ 0\end{pmatrix},\quad v_2=w_2=\begin{pmatrix} 0 \\ 1\end{pmatrix}
$$
so that
$$
B=v_1w_1^t+v_2w_2^t=\begin{pmatrix} 1\\0\end{pmatrix}\begin{pmatrix}1 & 0\end{pmatrix}+\begin{pmatrix}0\\ 1\end{pmatrix}\begin{pmatrix} 0 & 1\end{pmatrix}.
$$
Define
$$
J:=\left(\begin{array}{c|c}
E & C_{1,2}\\ \hline
0 & D_{1,2}
\end{array}\right)
$$
where
$$
E:=A+x_1v_1w_1^t+x_2v_2w_2^t=\begin{pmatrix} x_1 & 1\\ -1 & x_2\end{pmatrix},
$$
$$
C_{1,2}:=\begin{pmatrix} (y_{1,2}-x_1)v_1 & (y_{2,1}-x_2)v_2\end{pmatrix}=\begin{pmatrix}
y_{1,2}-x_1 & 0\\ 
0 & y_{2,1}-x_2
\end{pmatrix}
$$
and
$$
D_{1,2}:=\begin{pmatrix} y_{1,2}-z_{1,2} & y_{2,1}-z_{1,2}
\end{pmatrix}
$$
So we have
$$
J=\left(\begin{array}{cc|cc}
x_{1} & 1 & y_{1,2}-x_1 & 0\\
-1 & x_{2} & 0 & y_{2,1}-x_2\\ \hline
0 & 0 & y_{1,2}-z_{1,2} & y_{2,1}-z_{1,2}\end{array}\right).
$$
We can write
\begin{multline*}
J=\begin{pmatrix} 0 & 1 & 0 & 0\\
-1 & 0 & 0& 0\\
0 & 0 & 0& 0\end{pmatrix}+x_1\begin{pmatrix} 1  & 0 & -1 & 0\\
0 & 0 & 0 & 0\\
0 & 0 & 0 & 0\end{pmatrix}
+x_2\begin{pmatrix}
 0 & 0 & 0 & 0\\
 0 & 1 & 0 & -1\\
 0 & 0 & 0 & 0
 \end{pmatrix}+
 y_{1,2}\begin{pmatrix}
  0 & 0 & 1 & 0\\
  0 & 0 & 0 & 0\\
  0 & 0 & 1 & 0
  \end{pmatrix}+\\
  +y_{2,1}\begin{pmatrix}
  0 & 0 & 0 & 0\\
  0 & 0 & 0 & 1\\
  0 & 0 & 0 & 1
  \end{pmatrix}
  +
  z_{1,2}\begin{pmatrix}
  0 & 0 & 0 & 0\\
  0 & 0 & 0 &0\\
  0 & 0 & -1 & -1
  \end{pmatrix}.
  \end{multline*}
  \end{example}
\begin{example}
Suppose that
$$
A=\begin{pmatrix}
0& 3 & 4\\
0 & 0 & 5\\
0 & 0 & 0\end{pmatrix}, \quad
B=\begin{pmatrix}
1 & 0 & 0\\
2 & 1 & 0\\
3 & 2 & 1
\end{pmatrix}
$$
We can write
$$
B=\begin{pmatrix}
1 \\
2\\
3\end{pmatrix}\begin{pmatrix} 1 & 0 & 0\end{pmatrix}+
\begin{pmatrix}
0 \\
1\\
2\end{pmatrix}\begin{pmatrix}0 & 1 & 0\end{pmatrix}+\begin{pmatrix} 0\\0\\1\end{pmatrix}\begin{pmatrix}0 & 0 & 1\end{pmatrix}
$$
We have
$$
E=\begin{pmatrix}
x_1 & 3 & 4\\
2x_1 & x_2 & 5\\
3x_1 & 2x_2 & x_3
\end{pmatrix},
$$
\begin{multline*}
C=\begin{pmatrix}
C_{1,2} & C_{1,3} & C_{2,3}\end{pmatrix}=\\=\begin{pmatrix}
y_{1,2}-x_1 & 0 & y_{1,3}-x_1 & 0 & 0 & 0\\
2y_{1,2}-2x_1& y_{2,1}-x_2 & 2y_{1,3}-2x_1 & 0 & y_{2,3}-x_2 & 0\\
3y_{1,2}-3x_1 &  2y_{2,1}-2x_2 & 3y_{1,3}-3x_1 & y_{3,1}-x_3 & 2y_{2,3}-2x_2 & y_{3,2}-x_3
\end{pmatrix}
\end{multline*}
and
\begin{multline*}
D=\begin{pmatrix}
D_{1,2} & 0 & 0\\
0 & D_{1,3} & 0\\
0 & 0 & D_{2,3}\end{pmatrix}=\\
=\begin{pmatrix}
y_{1,2}-z_{1,2} & y_{2,1}-z_{1,2} & 0 & 0 & 0 &0\\
0 & 0 & y_{1,3}-z_{1,3} & y_{3,1}-z_{1,3} & 0 & 0\\
0 & 0 & 0 & 0 & y_{2,3}-z_{2,3} & y_{3,2}- z_{2,3}
\end{pmatrix}
\end{multline*}
So we have
$$
J=\left(
\begin{array}{ccc|cccccc}
x_1 & 3 & 4 & y_{1,2}-x_1 & 0 & y_{1,3}-x_1 & 0 & 0 & 0\\
2x_1 & x_2 & 5 & 2y_{1,2}-2x_1& y_{2,1}-x_2 & 2y_{1,3}-2x_1 & 0 & y_{2,3}-x_2 & 0\\
3x_1 & 2x_2 & x_3 & 3y_{1,2}-3x_1 &  2y_{2,1}-2x_2 & 3y_{1,3}-3x_1 & y_{3,1}-x_3 & 2y_{2,3}-2x_2 & y_{3,2}-x_3\\ \hline
0 & 0 & 0 &y_{1,2}-z_{1,2} & y_{2,1}-z_{1,2} & 0 & 0 & 0 &0\\
0 & 0 & 0 & 0 & 0 & y_{1,3}-z_{1,3} & y_{3,1}-z_{1,3} & 0 & 0\\
0 & 0 & 0 &0 & 0 & 0 & 0 & y_{2,3}-z_{2,3} & y_{3,2}- z_{2,3}
\end{array}\right)
$$
\end{example}

We will formulate the general rank minimization problem with upper indices, so that we can reserve lower indices for other uses. 
Suppose $A,B^{(1)},\dots,B^{(s)}\in \Mat_{m,n}(\F)$ and we 
want to minimize
$$
\rank(A+x^{(1)}B^{(1)}+\cdots+x^{(s)}B^{(s)}).
$$
over  all $x^{(1)},\dots,x^{(s)}\in \F$.
Inductively we construct a matrices by
$$
J_0(A)=A.
$$
and
\begin{multline*}
J_i(A,B^{(1)},\dots,B^{(i)},{\bf x}^{(1)},\dots,{\bf x}^{(i)}, {\bf y}^{(1)},\dots,{\bf y}^{(i)},{\bf z}^{(1)},\dots,{\bf z}^{(i)}):=\\
J(J_{i-1}(A,B^{(1)},\dots,B^{(i-1)},{\bf x}^{(1)},\dots,{\bf x}^{(i-1)}, {\bf y}^{(1)},\dots,{\bf y}^{(i-1)},{\bf z}^{(1)},\dots,{\bf z}^{(i-1)}), B^{(i)}, {\bf x}^{(i)},{\bf y}^{(i)},{\bf z}^{(i)}).
\end{multline*}
for $i=1,2,\dots,s$.
Inductively, we can find $x^{(s)},x^{(s-1)},\dots,x^{(1)}$ (in that order) such that
\begin{multline*}
\rank J_i(A+x^{(i+1)}B^{(i+1)}+\cdots+x^{(s)}B^{(s)},B^{(1)},\dots,B^{(i)},{\bf x}^{(1)},\dots,{\bf x}^{(i)}, {\bf y}^{(1)},\dots,{\bf y}^{(i)},{\bf z}^{(1)},\dots,{\bf z}^{(i)})\geq \\ \geq
\rank J_{i-1}(A+x^{(i)}B^{(i)}+\cdots+x^{(s)}B^{(s)}, {\bf x}^{(1)},\dots,{\bf x}^{(i-1)}, {\bf y}^{(1)},\dots,{\bf y}^{(i-1)},{\bf z}^{(1)},\dots,{\bf z}^{(i-1)}).
\end{multline*}
for all $i$.
Combining all the inequalities together gives
\begin{multline*}
\rank J_s(A,B^{(1)},\dots,B^{(s)},{\bf x}^{(1)},\dots,{\bf x}^{(s)}, {\bf y}^{(1)},\dots,{\bf y}^{(s)},{\bf z}^{(1)},\dots,{\bf z}^{(s)})\geq\\
\rank J_0(A+x^{(1)}B^{(1)}+\cdots+x^{(s)}B^{(s)})=\rank(A+x^{(1)}B^{(1)}+\cdots+x^{(s)}B^{(s)}).
\end{multline*}
If we set all the entries of ${\bf x}^{(i)},{\bf y}^{(i)},{\bf z}^{(i)}$ equal to $x^{(i)}$ for all $i$, then we have equality.
\begin{example}
Suppose that
$$
A=\begin{pmatrix}
-1 & 0 & 0\\
0 & -1 & 0\\
0 & 0 & -1
\end{pmatrix},
B^{(1)}=\begin{pmatrix}
0 & 1 & 0\\
-1 & 0 & 0\\
0 & 0 & 0
\end{pmatrix},
B^{(2)}=\begin{pmatrix}
0 & 0 & 1\\
0 & 0 & 0\\
-1 & 0 & 0
\end{pmatrix},
B^{(3)}=\begin{pmatrix}
0 & 0 & 0\\
0 & 0 & 1\\
0 & -1 & 0
\end{pmatrix}.
$$
so that
$$
A+x^{(1)}B^{(1)}+x^{(2)}B^{(2)}+x^{(3)}B^{(3)}=
\begin{pmatrix}
-1 & x^{(1)} & x^{(2)}\\
-x^{(1)} & -1 & x^{(3)}\\
-x^{(2)} & - x^{(3)} & -1.
\end{pmatrix}.
$$
We write
\begin{eqnarray*}
B^{(1)}&=&\begin{pmatrix} 1 \\ 0 \\0\end{pmatrix}\begin{pmatrix}0 & 1 & 0\end{pmatrix}+
\begin{pmatrix} 0\\ 1 \\ 0\end{pmatrix}\begin{pmatrix}-1 & 0 & 0\end{pmatrix},\\
B^{(2)}&=&\begin{pmatrix} 1\\ 0 \\0\end{pmatrix}\begin{pmatrix}0 & 0 & 1\end{pmatrix}+
\begin{pmatrix} 0\\ 0 \\ 1\end{pmatrix}\begin{pmatrix}-1 & 0 & 0\end{pmatrix},\\
B^{(3)}&=&\begin{pmatrix} 0\\ 1 \\0\end{pmatrix}\begin{pmatrix}0 & 0 & 1\end{pmatrix}+
\begin{pmatrix} 0\\ 0 \\ 1\end{pmatrix}\begin{pmatrix}0 & -1 & 0\end{pmatrix}.
\end{eqnarray*}
We get
$$
J_3=\left(\begin{array}{ccc|cc|cc|cc}
-1 & x^{(1)}_1 & x^{(2)}_1 & y_{1,2}^{(1)}-x^{(1)}_1 & 0 & y_{1,2}^{(2)}-x_1^{(2)} & 0 & 0 & 0\\
-x^{(1)}_2 & -1 & x^{(3)}_1 & 0 & y_{2,1}^{(1)}-x^{(1)}_2 &0 & 0& y_{1,2}^{(3)}-x_1^{(3)} & 0\\
-x^{(2)}_2 & -x^{(3)}_2 & -1 & 0 & 0 & 0 & y_{2,1}^{(2)}-x_2^{(2)} & 0 & y_{2,1}^{(3)}-x_2^{(3)}\\ \hline
0 & 0 & 0 & y_{1,2}^{(1)}-z_{1,2}^{(1)} & y_{2,1}^{(1)}-z_{1,2}^{(1)} & 0 & 0 & 0 & 0\\ \hline
0 & 0 & 0 & 0 & 0 & y_{1,2}^{(2)}-z_{1,2}^{(2)} & y_{2,1}^{(2)}- z_{1,2}^{(2)} & 0 & 0\\ \hline
0 & 0 & 0 & 0 & 0 & 0 & 0 & y_{1,2}^{(3)}-z_{1,2}^{(3)} & y_{2,1}^{(3)}-z_{1,2}^{(3)}
\end{array}\right).
$$
\end{example}
\section{Reduction from {\bf 1--RM} to {\bf LRMC}}\label{1RMtoLRMC}
We first deal with the Special Rank Minimization Problem with $s=1$. Suppose that $A,B$ are $n\times m$ matrices,
and that $B$ has rank $1$. We can write $B=vw^t$ with $v\in \F^n$ and $w\in \F^m$.

From 
$$
\begin{pmatrix} A+xvw^t & 0\\w^t & -1\end{pmatrix}=\begin{pmatrix} I & xv\\ 0 & 1\end{pmatrix}\begin{pmatrix} A & xv\\w^t & -1\end{pmatrix}
$$
follows that
\begin{multline*}
\rank(A+xB)+1=\rank(A+xvw^t)+1=
\rank\begin{pmatrix} A+xvw^t & 0\\ w^t & -1\end{pmatrix}=\\=
\rank\begin{pmatrix} A & xv\\ w^t & -1\end{pmatrix}=\rank(A'+xB'),
\end{multline*}
where
$$
A'=\begin{pmatrix} A& 0\\ w^t & -1\end{pmatrix},\quad B'=\begin{pmatrix} 0 & v\\0 & 0\end{pmatrix}.
$$

Note that $B'$ has at most 1 nonzero column.
If $B$ has at most 1 nonzero row, then $v$ will have only one nonzero entry and $B'$ will only have one nonzero entry. In that case we have
reduced to the case where $B$ has only one nonzero entry. So we have reduced a special rank minimization  problem to a matrix completion problem in this case.
\begin{example}
Suppose that 
$$
A=\begin{pmatrix} 1 & 0 \\ 0 & 1 \end{pmatrix},\quad B=\begin{pmatrix} 0 & 0 \\ 2 & 3\end{pmatrix}.
$$
We have
$$
A'=\begin{pmatrix} 1 & 0\ & 0 \\ 0 & 1 & 0\\
2 & 3 & -1
\end{pmatrix},\quad B'=\begin{pmatrix} 0 & 0 & 0\\
0 & 0 & 1\\
0 & 0 & 0
\end{pmatrix}
$$
So we have reduced the problem of minimizing the rank of 
$$
A+xB=\begin{pmatrix} 1 & 0 \\ 2x & 1+3x\end{pmatrix}
$$
to the matrix completion problem
$$
\begin{pmatrix}
1 & 0 & 0\\
0 & 1 & ?\\
2 & 3 & -1
\end{pmatrix}.
$$
The minimal rank of $A+xB$ is $1$, namely for $x=-\frac{1}{3}$. The smallest rank for the matrix completion problem is $2$. This is the case
when the missing entry is $-\frac{1}{3}$.
\end{example}

Suppose again that $A,B\in \Mat_{n,m}(\F)$ and $B=vw^t$ has rank $1$. 
By symmetry, we also have
\begin{multline*}
\rank(A+xB)+1=\rank(A+xvw^t)+1=
\rank\begin{pmatrix} A+xvw^t & v\\ 0 & -1\end{pmatrix}=\\=
\rank\begin{pmatrix} A & v\\ xw^t & -1\end{pmatrix}=\rank(A''+xB''),
\end{multline*}
where
$$
A''=\begin{pmatrix} A& v\\ 0 & -1\end{pmatrix},\quad B'=\begin{pmatrix} 0 & 0\\w^t & 0\end{pmatrix}.
$$

Note that $B'$ has at most 1 nonzero row.
If $B$ has at most 1 nonzero column, then $w^t$ will have only one nonzero entry and $B''$ will only have one nonzero entry. In that case we have
reduced the special matrix minimization problem to a matrix completion problem.

In the general case, where $B=vw^t$ has more than 1 nonzero row and more than 1 nonzero column, we can use the first construction to obtain matrices $A',B'$
such that 
$$
\rank(A+xB)+1=\rank(A'+xB')
$$
and $B'$ has at most one nonzero column. Then we can use the second contruction to obtain matrices $A'',B''$ such that
$$
\rank(A'+xB')+1=\rank(A''+xB'')
$$
such that $B''$ has only one nonzero entry. Note that
$$
\rank(A+xB)+2=\rank(A''+xB'').
$$
So we have reduced the problem of minimizing the rank of $A+xB$ to the matrix completion problem.
\begin{example}
Suppose that
$$
A=\begin{pmatrix}
1 & 0\\
0 & 1
\end{pmatrix},\quad B=\begin{pmatrix} 1 & 2\\ 3 & 6\end{pmatrix}=\begin{pmatrix} 1\\3\end{pmatrix}\begin{pmatrix}1 & 2\end{pmatrix}.
$$
If we define
$$
A'=\begin{pmatrix}
1 & 0 & 0 \\
0 & 1 & 0\\
1 & 2 & -1
\end{pmatrix},\quad B'=\begin{pmatrix}
0 & 0 & 1\\
0 & 0 & 3\\
0 & 0 & 0
\end{pmatrix}=\begin{pmatrix}1\\3\\0\end{pmatrix}\begin{pmatrix}0 & 0 & 1\end{pmatrix}
$$
then
$$
\rank(A+xB)+1=\rank(A'+xB').
$$
Now, define
$$
A''=\begin{pmatrix}
1 & 0 & 0 & 1\\
0 & 1 & 0 & 3\\
1 & 2 & -1 & 0\\
0 & 0 & 0 & -1
\end{pmatrix},\quad
B''=\begin{pmatrix}
0 & 0 & 0 & 0\\
0 & 0 & 0 & 0\\
0 & 0 & 0 & 0\\
0 & 0 & 1 & 0
\end{pmatrix}.
$$
Then we have
$$
\rank(A+xB)+2=\rank(A''+xB'').
$$
So we have reduced the problem of minimizing the rank of $A+xB$ to the low rank matrix completion problem
$$
\begin{pmatrix}
1 & 0 & 0 & 1\\
0 & 1 & 0 & 3\\
1 & 2 & -1 & 0\\
0 & 0 & ? & -1
\end{pmatrix}.
$$
The smallest possible rank of this matrix is $3$, when the missing entry is $-\frac{1}{7}$.
\end{example}

Suppose that $A\in \Mat_{n,m}(\F)$ and that $B_1,\dots,B_s\in\Mat_{n,m}(\F)$ have rank $1$.
By repeated use of the construction above, we can reduce the problem of minimizing 
$$
\rank(A+x_1B_1+\cdots+x_sB_s) 
$$
to the Low Rank Matrix Completion Problem.
\begin{example}
Consider again Example~\ref{ex:2}. We defined
$$
A=\begin{pmatrix}
0 & 1\\
-1 & 0
\end{pmatrix},\quad B=\begin{pmatrix} 1 & 0\\
0 & 1\end{pmatrix}
$$
and
we reduced the problem of minimizing
$$
\rank(A+xB)
$$
to the problem of minimizing the rank of 
\begin{multline*}
J=\begin{pmatrix} 0 & 1 & 0 & 0\\
-1 & 0 & 0& 0\\
0 & 0 & 0& 0\end{pmatrix}+x_1\begin{pmatrix} 1  & 0 & -1 & 0\\
0 & 0 & 0 & 0\\
0 & 0 & 0 & 0\end{pmatrix}
+x_2\begin{pmatrix}
 0 & 0 & 0 & 0\\
 0 & 1 & 0 & -1\\
 0 & 0 & 0 & 0
 \end{pmatrix}+
 y_{1,2}\begin{pmatrix}
  0 & 0 & 1 & 0\\
  0 & 0 & 0 & 0\\
  0 & 0 & 1 & 0
  \end{pmatrix}+\\
  +y_{2,1}\begin{pmatrix}
  0 & 0 & 0 & 0\\
  0 & 0 & 0 & 1\\
  0 & 0 & 0 & 1
  \end{pmatrix}
  +
  z_{1,2}\begin{pmatrix}
  0 & 0 & 0 & 0\\
  0 & 0 & 0 &0\\
  0 & 0 & -1 & -1
  \end{pmatrix}.
  \end{multline*}
This can be reduced to the Low Rank Matrix Completion Problem:
$$
\left(\begin{array}{cccc|c|c|c|c|c}
0 & 1 & 0 & 0 & ? & 0 & 1 & 0 & 0\\
-1 & 0 & 0 & 0 & 0 & ? & 0 & 1 & 0\\
0 & 0 & 0 & 0 & 0 & 0 & 1 & 1 & ?\\ \hline
1 & 0 & -1 & 0 & -1 & 0 & 0 & 0 & 0\\ \hline
0 & 1 & 0 & -1 & 0 & -1 & 0 & 0 & 0\\ \hline
0 & 0 & ? & 0 & 0 & 0 & -1 & 0 & 0\\ \hline
0 & 0 & 0 & ? & 0 & 0 & 0 & -1 & 0\\ \hline
0 & 0 & -1 & -1 & 0 & 0 & 0 & 0 & -1
\end{array}\right).
$$
Over the field $\R$, the smallest possible rank of this matrix is $2+5=7$, for example when all the missing entries are 0. But over $\C$, the minimal rank 
of this matrix is $1+5=6$. This is the case when all the missing entries are equal to the imaginary number $i$.

\end{example}

%\begin{multline*}
%\rank(A+xB)=\rank(A+xvw^t)=\rank\begin{pmatrix} A & v\\ xw^t & -1\end{pmatrix}=\\
%\rank\left(\begin{pmatrix} A& v\\ 0 &-1\end{pmatrix}+x\begin{pmatrix} 0 & 0\\ xw^t & 0\end{pmatrix}\right)=\rank(A'+xB')
%\end{multline*}
%where $B'$ has at most 1 nonzero row. If $B$ has at most 1 nonzero column, then $w$ will have at most one nonzero entry, so $B'$ will have at most one nonzero entry.
%Combining both gives
%$$
%\rank(A+xB)+2=\rank\left(\begin{pmatrix} A& 0\\ w^t&-1\end{pmatrix}+x\begin{pmatrix} 0 & v\\ 0 & 0\end{pmatrix}\right)+1=
%\rank\left(\begin{pmatrix} A& 0 & v\\ w^t&-1& 0 \\ 0 & x & -1\end{pmatrix}\right)
%$$
%So we have reduced minimizing $\rank(A+xB)$ over all $x\in \F$ to a matrix completion problem.

\section{Reduction from {\bf 1--RM}  to {\bf 3-TR}}\label{1RMto3TR}
Suppose that $A=(a_{i,j})\in \Mat_{n,m}(\F)$  and that $B_1,\dots,B_s\in \Mat_{n,m}(\F)$ have rank $1$ and are linearly independent.
We can write $B_i=v_iw_i^t$ where $v_i\in \F^n$ and $w_i\in \F^m$ for all $i$.
We will identify $\Mat_{n,m}(\F)$ with $\F^n\otimes \F^m$. Then $B_i$ will be identified with $v_i\otimes w_i\in \F^n\otimes \F^m$.
We define a third order tensor $T\in \F^n\otimes \F^m\otimes \F^{s+1}$
by
\begin{equation}\label{eq:Ahat}
T=A\otimes e_{s+1}+\sum_{k=1}^s B_k\otimes e_k=
\sum_{i=1}^n\sum_{j=1}^m a_{i,j} (e_i\otimes e_j\otimes e_{s+1})+\sum_{k=1}^s v_k\otimes w_k\otimes e_k.
\end{equation}
Let $l=\rank T$. We can write
$$
T=\sum_{j=1}^l a_j\otimes b_j\otimes c_j.
$$
\begin{lemma}
The vectors $c_1,\dots,c_l,e_{s+1}$ span $\F^{s+1}$.
\end{lemma}
\begin{proof}
For every $i$, choose a linear function $h_i:\F^n\otimes \F^m\to \F$ such that $h_i(B_i)=1$
and $h_i(B_j)=0$ if $j\neq i$.
Applying $h_i\otimes \id:\F^n\otimes \F^m\otimes \F^{s+1}\to \F^{s+1}$ to the tensor $T$ gives
$$
\sum_{j=1}^l h_i(a_j\otimes b_j) c_j=(h_i\otimes \id)(T)=e_i+h_i(A) e_{s+1}.
$$
This shows that $e_i$ lies in the span of $c_1,\dots,c_l,e_{s+1}$ for all $i$. So $c_1,\dots,c_l,e_{s+1}$
span the vector space $\F^{s+1}$.
\end{proof}
After rearranging $c_1,\dots,c_l$, we may assume that $c_1,c_2,\dots,c_{s},e_{s+1}$ is a basis of $\F^{s+1}$.
There exists a linear function $f:\F^{s+1}\to \F$ such that $f(c_1)=\cdots=f(c_s)=0$ and $f(e_{s+1})=1$.
We apply $\id\otimes \id\otimes f:\F^n\otimes \F^m\otimes \F^{s+1}\to \F^n\otimes \F^m$ to $T$:
$$
A+\sum_{k=1}^s f(e_k)B_k=(\id\otimes\id\otimes f)(T)=\sum_{j=1}^l f(c_j) a_j\otimes b_j=\sum_{j=s+1}^l f(c_j)a_j\otimes b_j.
$$
The resulting matrix clearly has rank at most $l-s$.
\begin{theorem}
The smallest possible rank of
$$
A+x_1B_1+\cdots+x_sB_s
$$
over all $(x_1,\dots,x_s)\in \F^s$ is $l-s$. This minimum is attained when $x_i=f(e_i)$ for $i=1,2,\dots,s$.
\end{theorem}
\begin{proof}
Suppose that $x_1,\dots,x_s\in \F$ and $\rank(A+x_1B_1+\cdots+x_sB_s)=r$. We have to show that $r\geq l-s$.
We can write
$$
A+x_1B_1+\dots+x_sB_s=C_1+\cdots+C_r,
$$
where $C_1,\dots,C_r$ are matrices of rank $1$.
We have 
\begin{multline*}
T=A\otimes e_{s+1}+\sum_{k=1}^s B_k\otimes e_k=
(A+\sum_{k=1}^s x_k B_k)\otimes e_{s+1}-\sum_{k=1}^s x_k B_k\otimes e_{s+1}+\sum_{k=1}^s B_k\otimes e_k=\\=
\sum_{j=1}^r C_j\otimes e_{s+1}+\sum_{k=1}^s B_k\otimes (e_k -x_k e_{s+1})
\end{multline*}
We have written $T$ as a sum of $r+s$ pure tensors. So $l=\rank T\leq r+s$.

\end{proof}

\begin{remark}
The reduction in this section also easily generalizes to the higher order analog of matrix completion: {\em tensor completion}. 
Let $V=\F^{n_1}\otimes \cdots\otimes \F^{n_d}$ and suppose that $A\in V$, and $B_1,\dots,B_s$ are pure  tensors that are linearly dependent.
Define
$$
T=A\otimes e_{s+1}+\sum_{k=1}^s B_k\otimes e_k\in V\otimes \F^{s+1}=\F^{n_1}\otimes \cdots\otimes \F^{n_d}\otimes \F^{s+1}.
$$
Suppose that $T$ has rank $l$. Then we can write
$$
T=\sum_{j=1}^l D_j\otimes c_j
$$
where $D_1,\dots,D_l\in V$ are pure tensors. The vectors $c_1,\dots,c_l,e_{s+1}$ span $\F^{s+1}$.
After rearranging, we may assume that $c_1,\dots,c_s,e_{s+1}$ is a basis. Define a linear map $f:\F^{s+1}\to \F$ by $f(c_1)=\cdots=f(c_s)=0$ and $f(e_{s+1})=1$.
Now the smallest possible rank of the tensor
$$
A+\sum_{i=1}^s x_iB_i
$$
over all $x_1,\dots,x_s\in \F$ is $l-s$. Equality is attained when we take $x_i=f(e_i)$ for $i=1,2,\dots,s$, and we have a decomposition
$$
A+\sum_{i=1}^s f(e_i)B_i=\sum_{j=s+1}^l f(c_j)D_j.
$$

\end{remark}

\begin{example}
Consider the matrix completion problem over $\C$ for the matrix
$$
\begin{pmatrix}
1 & \varepsilon\\
? & 1
\end{pmatrix}
$$
where $\varepsilon\in \C$ is some constant. If $\varepsilon\neq 0$, then the matrix can be completed to a rank $1$ matrix:
$$
\begin{pmatrix}
1 & \varepsilon\\
\varepsilon^{-1} & 1
\end{pmatrix}.
$$
But if $\varepsilon=0$, then the rank of any completion will be 2. We can view this as a rank minimization problem 
where 
$$
A=\begin{pmatrix}
1 & \varepsilon \\
0 & 1
\end{pmatrix}=e_1\otimes (e_1+\varepsilon e_2)+e_2\otimes e_2\in \C^2\otimes \C^2
$$
and 
$$
B=\begin{pmatrix}
0 & 0\\
1 & 0
\end{pmatrix}=e_2\otimes e_1\in \C^2\otimes \C^2,
$$
and we want to minimize the rank of $A+xB$. Let
$$
T=
e_1\otimes (e_1+\varepsilon e_2)\otimes e_2+e_2\otimes e_2\otimes e_2+e_2\otimes e_1\otimes e_1.
$$
For $\varepsilon\neq 0$, this tensor has rank $2$:
$$
T=(e_1+\varepsilon^{-1} e_2)\otimes (e_1+\varepsilon e_2)\otimes e_2+e_2\otimes e_1\otimes (e_1-\varepsilon^{-1}e_2).
$$
 But for $\varepsilon=0$ the tensor has rank $3$.
This is the {\em border rank phenomenon}: a rank 3 tensor can be the limit of rank 2 tensors.
Of course, for matrices (order 2 tensors), this is not possible.
\end{example}

\section{Reduction from {\bf 3--TR} to {\bf RM}}\label{3TRtoRM}
Suppose that $T\in \F^p\otimes \F^q\otimes \F^r$. We would like to find pure tensors $v_1,\dots,v_k$ such that $T=v_1+\cdots+v_k$
and $k$ is minimal. In this section we reduce this problem to rank minimization. 
First we need some theoretical results.
\begin{lemma}
We can write  $T=\sum_{i=1}^r S_i\otimes e_i$ where $S_1,\dots,S_r\in \F^p\otimes \F^q\cong \Mat_{p,q}(\F)$.
If $\rank T=l$ then there exist rank $1$ matrices $U_1,\dots,U_l\in \Mat_{p,q}(\F)$ such that $S_1,\dots,S_r$ lie in the 
span of $U_1,\dots,U_l$.
\end{lemma}
\begin{proof}
We can write $T=\sum_{i=1}^l U_i\otimes f_i$. Let $\pi_i:\F^r\to \F$ be the projection on the $i$-th coordinate, and consider the projection
$$
\id\otimes\id\otimes \pi_i:\F^p\otimes \F^q\otimes \F^r\to \F^p\otimes \F^q.
$$
We have $(\id\otimes\id\otimes \pi_i)(T)=S_i=\sum_{j=1}^l \pi_i(f_j)U_j$.
\end{proof}
\begin{lemma}
Suppose that $a_1,\dots,a_k\in \F^p$, $b_1,\dots,b_k\in \F^q$ and $\lambda_1,\dots,\lambda_k\in \F$ and let
$$
A=\begin{pmatrix}a_1 & \cdots & a_k\end{pmatrix},\ 
B=\begin{pmatrix} b_1 & \cdots & b_k\end{pmatrix},\ \Lambda=\begin{pmatrix} \lambda_1 & & \\ & \ddots & \\ && \lambda_k\end{pmatrix},\ 
C=\begin{pmatrix} S & A & 0\\  0 & I _k& -\Lambda \\  B^t & 0 & I_k\end{pmatrix}. 
$$
Then we have
$$
\rank C=2k+\rank\Big(\textstyle S-\sum_{j=1}^k \lambda_j a_jb_j^t\Big).
$$
\end{lemma}
\begin{proof}
We perform elementary row operations on the matrix $C$:
$$
\begin{pmatrix}
I_p & 0 & 0\\
0 & I_k & \Lambda\\
0 & 0 & I_k
\end{pmatrix} C=\begin{pmatrix}
I_p & 0 & 0\\
0 & I_k & \Lambda\\
0 & 0 & I_k
\end{pmatrix} 
\begin{pmatrix} S & A & 0\\  0 & I _k& -\Lambda \\  B^t & 0 & I_k\end{pmatrix}
=\begin{pmatrix}
S  & A & 0\\
\Lambda B^t & I_k & 0\\
B^t & 0 & I_k
\end{pmatrix}
$$
$$
\begin{pmatrix}
I_p & -A & 0\\
0 & I_k &0 \\
0 & 0 & I_k
\end{pmatrix}
\begin{pmatrix}
S  & A& 0\\
\Lambda B^t & I_k & 0\\
B^t & 0 & I_k
\end{pmatrix}=
\begin{pmatrix}
S-A\Lambda B^t  & 0 & 0\\
\Lambda B^t & I_k & 0\\
B^t & 0 & I_k
\end{pmatrix}
$$
So we have
$$
\rank C=\rank\begin{pmatrix}
S-A\Lambda  B^t  & 0 & 0\\
\Lambda B^t & I_k & 0\\
B^t & 0 & I_k
\end{pmatrix}
=2k+\rank(S-A\Lambda B^t)=2k+\rank\Big(\textstyle S-\sum_{j=1}^k \lambda_ja_jb_j^t\Big).
$$
\end{proof}
From now on, assume that $k\geq \rank T$. Let $A\in \Mat_{p,k}(\F)$ and $B\in \Mat_{q,k}(\F)$ be matrices with indeterminate entries.
For every $i$ with $1\leq i\leq r$ define
$$
\Lambda_i=\begin{pmatrix} \lambda_{i,1} & & \\ & \ddots & \\ & & \lambda_{i,k}
\end{pmatrix},
$$
where $\lambda_{i,1},\dots,\lambda_{i,k}$ are indeterminates.
Let ${\bf a}$ be a list of the entries of $A$, ${\bf b}$ a list of the entries of $B$ and $\lambda$ a list of all $\lambda_{i,j}$.
Define
$$
C_i({\bf a},{\bf b},\lambda)=\begin{pmatrix} S_i & A & 0 \\ 0 & I_k & -\Lambda_i\\
B^t  & 0 & I_k\end{pmatrix}
$$
and 
$$
U_i({\bf a},{\bf b},\lambda)=S_i-\sum_{j=1}^k\lambda_{i,j}a_jb_j^t
$$
where $a_1,\dots,a_k$ are the columns of $A$ and $b_1,\dots,b_k$ are the columns of $B$.
We have
$$
\rank C_i({\bf a}, {\bf b},\lambda)=2k+\rank U_i({\bf a},{\bf b},\lambda).
$$

We can write
\begin{multline*}
T=\sum_{i=1}^r S_i\otimes e_i=\sum_{i=1}^rU_i({\bf a},{\bf b},\lambda)\otimes e_i +\sum_{i=1}^r \sum_{j=1}^k \lambda_{i,j}a_jb_j^t\otimes c_i
=\\=\sum_{i=1}^r U_i({\bf a},{\bf b},\lambda)\otimes e_i+\sum_{j=1}^k a_jb_j^t\otimes (\sum_{i=1}^r \lambda_{i,j}c_i).
\end{multline*}
It follows that
\begin{equation}\label{eq:rankTupperbound}
\rank T\leq \sum_{i=1}^r \rank U_i({\bf a},{\bf b},\lambda)+\sum_{j=1}^k \rank(a_jb_j^t).
\end{equation}

Define 
$$
D_j({\bf a},{\bf b})=\begin{pmatrix} a_j\\ b_j\end{pmatrix}.
$$
Then we have
$$
\rank D_j({\bf a},{\bf b})\geq \rank(a_jb_j^t)
$$
with equality if $a_i$ and $b_j$ are both zero or both nonzero.
Define
$$
E({\bf a},{\bf b},\lambda)=\begin{pmatrix}
C_1({\bf a},{\bf b},\lambda) & & &&&\\ 
& \ddots &&&&\\
&& C_r({\bf a},{\bf b},\lambda) & & &\\
& & & D_1({\bf a},{\bf b}) & &\\
& & & & \ddots &\\
& & & & & D_k({\bf a},{\bf b})
\end{pmatrix}.
$$
\begin{theorem}
We have
$$
\min_{{\bf a},{\bf b},\lambda} \rank E({\bf a},{\bf b},\lambda)=2kr+\rank T.
$$
\end{theorem}
\begin{proof}

%\begin{lemma}
%We have the following inequality
%$$
%\sum_{i=1}^r\rank(C_i({\bf a},{\bf b}))+\sum_{j=1}^k \rank(a_ib_j^t)\geq \rank(T).
%$$
%\end{lemma}
%\begin{proof}
%Then $U_i$ lies in a subspace spanned by rank $1$ matrices whose dimension is $\rank(U_i)$.
%So $U_1,\dots,U_r$ lie in a subspace $V$ of dimension $\leq \sum_{i=1}^r \rank(U_i)$ that is spanned by rank 1 matrices.
%Let $W$ be the span of all $a_ib_i^t$, $i=1,2,\dots,r$. The dimension of $W$ is at most $\sum_{j=1}^k\rank(a_jb_j^t)$. 
%Now $V+W$ contains $S_1,\dots,S_r$, and its dimension is at most $\sum_{i=1}^r \rank(U_i)+\sum_{j=1}^k\rank(a_ib_j^t)$.
%\end{proof}
From (\ref{eq:rankTupperbound}) follows that
\begin{multline*}
\rank E({\bf a},{\bf b},\lambda)=\sum_{i=1}^r \rank(C_i({\bf a},{\bf b},\lambda) +\sum_{j=1}^k \rank D_j({\bf a},{\bf b}) \geq\\ \geq
 2kr+\sum_{i=1}^r \rank U_i({\bf a},{\bf b},\lambda)+\sum_{j=1}^k \rank(a_jb_j^t)
\geq 2kr+\rank T.
\end{multline*}
We have to show now that we can have equality for some choice of ${\bf a},{\bf b},\lambda$.
Suppose that $\rank T=l$. Then we can write $T=\sum_{j=1}^l a_j\otimes b_j\otimes c_j$ with $l\leq k$. Define $a_j=0$ and $b_j=0$
for $j=l+1,\dots,k$. If we apply the map $$
\id\otimes\id\otimes \pi_i:\F^p\otimes \F^q\otimes \F^r\to \F^p\otimes \F^q.
$$
to $T$, we get
$$
S_i=\id\otimes\id \otimes \pi_i(T)=\sum_{j=1}^l \pi_i(c_j)a_jb_j^t
$$
Define $\lambda_{i,j}=\pi_i(c_j)$. Then we have $U_i({\bf a},{\bf b},\lambda)=0$ for $i=1,2,\dots,r$. 
It follows that
\begin{multline*}
\rank E({\bf a},{\bf b},\lambda)=\sum_{i=1}^r \rank C_i({\bf a},{\bf b},\lambda)+\sum_{j=1}^k \rank D_k({\bf a},{\bf b})=\\=\Big(\sum_{i=1}^r 2k\Big)+l=rk+l=2rk+\rank T.
\end{multline*}

\end{proof}

Suppose that $\rank E({\bf a},{\bf b},\lambda)$ is minimal over all choices for ${\bf a},{\bf b},\lambda$.
We can write
\begin{multline*}
T=\sum_{i=1}^r S_i\otimes e_i=\sum_{i=1}^r U_i({\bf a},{\bf b},\lambda)\otimes e_i+\sum_{i=1}^r (\sum_{j=1}^k \lambda_{i,j} a_jb_j^t)\otimes e_i=\\=
\sum_{i=1}^r U_i({\bf a},{\bf b},\lambda)\otimes e_i+\sum_{j=1}^k a_jb_j^t\otimes \Big(\textstyle\sum_{i=1}^k \lambda_{i,j}e_i\Big).
\end{multline*}
If $d_i=\rank U_i({\bf a},{\bf b},\lambda)$ then we can write $U_i({\bf a},{\bf b},\lambda)$ as sum of $d_i$ rank $1$ matrices.
Since 
$$
\sum_{i=1}^r\rank U_i({\bf a},{\bf b},\lambda)+\sum_{i=1}^k \rank(a_jb_j^t)\leq  \sum_{i=1}^r \rank C_i({\bf a},{\bf b},\lambda)+\sum_{j=1}^k\rank D_j({\bf a},{\bf b})-2kr=l,
$$
we have written $T$ as a sum of $l=\rank T$ pure tensors.

The matrix $E({\bf a},{\bf b},\lambda)$ has $(p+2k)r+(p+q)k$ rows, and $(q+2k)r+k$ columns. There are $(p+q+k)r$ variables. 
We can write
$$
E({\bf a},{\bf b},\lambda)=A+\sum_{i,j}a_{i,j}F_{i,j}+\sum_{i,j}b_{i,j} G_{i,j}+\sum_{i,j}\lambda_{i,j}H_{i,j}
$$
Here $F_{i,j}$ and $G_{i,j}$ have rank $\leq r+1$ because the variables $a_{i,j}$ and $b_{i,j}$ appear  $r+1$ times.
The matrices $H_{i,j}$ have rank $1$ because they only have 1 nonzero entry.

\begin{example}
Suppose that $T=\sum_{i,j,k} a_{i,j,k} e_i\otimes e_j\otimes e_k$ is a tensor in $\F^2\otimes \F^2\otimes \F^2$.
Identifying $\F^2\otimes \F^2\otimes \F^2$ with $\Mat_{2,2}(\F)\otimes \F^2$ gives us
$$
T=\begin{pmatrix}
t_{1,1,1} & t_{1,2,1}\\
t_{2,1,1} & t_{2,2,1}
\end{pmatrix}\otimes e_1+\begin{pmatrix}
t_{1,1,2} & t_{1,2,2}\\
t_{2,1,2} & t_{2,2,2}
\end{pmatrix}\otimes e_2=S_1\otimes e_1+S_2\otimes e_2.
$$
The largest possible rank of a tensor $T$ in $\F^2\otimes \F^2\otimes \F^2$ is $3$. 
So let us take $k=3$.
Define 
$$
A=\begin{pmatrix}
a_{1,1} & a_{1,2} & a_{1,3}\\
a_{2,1} & a_{2,2} & a_{2,3}
\end{pmatrix},
B=\begin{pmatrix}
b_{1,1} & b_{1,2} & b_{1,3} \\
b_{1,2} & b_{2,2} & b_{2,3}
\end{pmatrix}.
$$
For $i=1,2$, define
$$
\Lambda_i=\begin{pmatrix}
\lambda_{i,1} & 0 & 0\\
0 & \lambda_{i,2} & 0 \\
0 & 0 & \lambda_{i,3}\end{pmatrix},
$$
$$
C_i=\begin{pmatrix}
S_i & A & 0\\
0 & I_3 & \Lambda_i\\
B^t & 0 & I_3
\end{pmatrix}=\left(\begin{array}{cc|ccc|ccc}
t_{1,1,i} & t_{1,2,i} & a_{1,1} & a_{1,2} & a_{1,3} & 0 & 0 & 0\\
t_{2,1,i} & t_{2,2,i} & a_{2,1} & a_{2,2} & a_{2,3} & 0 & 0 & 0\\ \hline
0 & 0 & 1 &0 & 0 & \lambda_{i,1} & 0 & 0 \\
0 &0 & 0 & 1 & 0 & 0 & \lambda_{i,2} & 0\\
0 & 0 & 0 & 0 & 1 & 0 & 0 & \lambda_{i,3} \\ \hline
b_{1,1} & b_{2,1} & 0 & 0 & 0 & 1 & 0 & 0\\
b_{1,2} & b_{2,2} & 0 & 0 & 0 & 0 & 1 & 0\\
b_{1,3} & b_{2,3} & 0 & 0 & 0 & 0 & 0 & 1
\end{array}\right)
$$
$$
D_j=\begin{pmatrix}
a_{1,j}\\
a_{2,j}\\
b_{1,j}\\
b_{2,j}
\end{pmatrix}
$$
So finally, we have
\begin{multline*}
E({\bf a},{\bf b},\lambda)=\begin{pmatrix}
C_1 & 0& 0& 0 & 0\\
0 & C_2 & 0 &0  &0\\
0 & 0 & D_1 & 0 & 0\\
0 & 0 & 0 & D_2 & 0\\
0 & 0 & 0 & 0 & D_3
\end{pmatrix}=\\
\setlength\arraycolsep{3pt}
\left(\begin{array}{cccccccc|cccccccc|c|c|c}
t_{1,1,1} & t_{1,2,1} & a_{1,1} & a_{1,2} & a_{1,3} & 0 & 0 & 0      & 0 & 0 & 0 & 0 & 0 & 0 & 0 & 0 & 0 & 0 & 0\\
t_{2,1,1} & t_{2,2,1} & a_{2,1} & a_{2,2} & a_{2,3} & 0 & 0 & 0  & 0 & 0 & 0 & 0 & 0 & 0 & 0 & 0 & 0 & 0 & 0\\
0 & 0 & 1 &0 & 0 & \lambda_{1,1} & 0 & 0  & 0 & 0 & 0 & 0 & 0 & 0 & 0 & 0 & 0 & 0 & 0 \\
0 &0 & 0 & 1 & 0 & 0 & \lambda_{1,2} & 0    & 0 & 0 & 0 & 0 & 0 & 0 & 0 & 0 & 0 & 0 & 0\\
0 & 0 & 0 & 0 & 1 & 0 & 0 & \lambda_{1,3}  & 0 & 0 & 0 & 0 & 0 & 0 & 0 & 0 & 0 & 0 & 0\\ 
b_{1,1} & b_{2,1} & 0 & 0 & 0 & 1 & 0 & 0  & 0 & 0 & 0 & 0 & 0 & 0 & 0 & 0 & 0 & 0 & 0\\
b_{1,2} & b_{2,2} & 0 & 0 & 0 & 0 & 1 & 0 & 0 & 0 & 0 & 0 & 0 & 0 & 0 & 0 & 0 & 0 & 0\\
b_{1,3} & b_{2,3} & 0 & 0 & 0 & 0 & 0 & 1  & 0 & 0 & 0 & 0 & 0 & 0 & 0 & 0 & 0 & 0 & 0 \\ \hline
0 & 0 & 0 & 0 & 0 & 0 & 0 & 0 &    t_{1,1,2} & t_{1,2,2} & a_{1,1} & a_{1,2} & a_{1,3} & 0 & 0 & 0    & 0 & 0 & 0\\
0 & 0 & 0 & 0 & 0 & 0 & 0 & 0 &t_{2,1,2} & t_{2,2,2} & a_{2,1} & a_{2,2} & a_{2,3} & 0 & 0 & 0& 0 & 0 & 0\\ 
0 & 0 & 0 & 0 & 0 & 0 & 0 & 0 &0 & 0 & 1 &0 & 0 & \lambda_{2,1} & 0 & 0& 0 & 0 & 0 \\
0 & 0 & 0 & 0 & 0 & 0 & 0 & 0 &0 &0 & 0 & 1 & 0 & 0 & \lambda_{2,2} & 0& 0 & 0 & 0\\
0 & 0 & 0 & 0 & 0 & 0 & 0 & 0 &0 & 0 & 0 & 0 & 1 & 0 & 0 & \lambda_{2,3}& 0 & 0 & 0 \\ 
0 & 0 & 0 & 0 & 0 & 0 & 0 & 0 &b_{1,1} & b_{2,1} & 0 & 0 & 0 & 1 & 0 & 0& 0 & 0 & 0\\
0 & 0 & 0 & 0 & 0 & 0 & 0 & 0 &b_{1,2} & b_{2,2} & 0 & 0 & 0 & 0 & 1 & 0& 0 & 0 & 0\\
0 & 0 & 0 & 0 & 0 & 0 & 0 & 0 &b_{1,3} & b_{2,3} & 0 & 0 & 0 & 0 & 0 & 1& 0 & 0 & 0\\ \hline
0 & 0 & 0 & 0 & 0 & 0 & 0 & 0 & 0 & 0 & 0 & 0 & 0 & 0 & 0 & 0 & a_{1,1} & 0 & 0\\
0 & 0 & 0 & 0 & 0 & 0 & 0 & 0 & 0 & 0 & 0 & 0 & 0 & 0 & 0 & 0 & a_{2,1} & 0 & 0\\
0 & 0 & 0 & 0 & 0 & 0 & 0 & 0 & 0 & 0 & 0 & 0 & 0 & 0 & 0 & 0 & b_{1,1} & 0 & 0\\
0 & 0 & 0 & 0 & 0 & 0 & 0 & 0 & 0 & 0 & 0 & 0 & 0 & 0 & 0 & 0 & b_{2,1} & 0 & 0\\ \hline
0 & 0 & 0 & 0 & 0 & 0 & 0 & 0 & 0 & 0 & 0 & 0 & 0 & 0 & 0 & 0 & 0  &  a_{1,2} & 0\\
0 & 0 & 0 & 0 & 0 & 0 & 0 & 0 & 0 & 0 & 0 & 0 & 0 & 0 & 0 & 0 & 0  &  a_{2,2} & 0\\
0 & 0 & 0 & 0 & 0 & 0 & 0 & 0 & 0 & 0 & 0 & 0 & 0 & 0 & 0 & 0 & 0  &  b_{1,2} & 0\\
0 & 0 & 0 & 0 & 0 & 0 & 0 & 0 & 0 & 0 & 0 & 0 & 0 & 0 & 0 & 0 & 0  &  b_{2,2} & 0\\ \hline
0 & 0 & 0 & 0 & 0 & 0 & 0 & 0 & 0 & 0 & 0 & 0 & 0 & 0 & 0 & 0 & 0  &  0 & a_{1,3}\\
0 & 0 & 0 & 0 & 0 & 0 & 0 & 0 & 0 & 0 & 0 & 0 & 0 & 0 & 0 & 0 & 0  &  0 & a_{2,3}\\
0 & 0 & 0 & 0 & 0 & 0 & 0 & 0 & 0 & 0 & 0 & 0 & 0 & 0 & 0 & 0 & 0  &  0 & b_{1,3}\\
0 & 0 & 0 & 0 & 0 & 0 & 0 & 0 & 0 & 0 & 0 & 0 & 0 & 0 & 0 & 0 & 0  &  0 & b_{2,3}\end{array}\right).
\end{multline*}
We have
$$
\min_{{\bf a},{\bf b}, \lambda} E({\bf a},{\bf b},\lambda)=12+\rank T.
$$

\end{example}

    \section{Reduction from ${\bf TR}$ to ${\bf RM}$}\label{TRtoRM}
    For a matrix $A\in \Mat_{m,n}(\F)$, we define its {\em padding} as the 
    partially filled matrix
    $$
    \pad(A)=
    \begin{pmatrix}
    A & ?\\
    ? & 1
    \end{pmatrix}
    $$
    \begin{lemma}
  We have
  $  \minrank \pad(A)=\max\{1,\rank A \}$.
    \end{lemma}
    \begin{proof}
    It is clear that $\rank \pad(A)\geq \max\{1,\rank A\}$, and we have equality if $A=0$.  Suppose that $A$ is nonzero and let $v$ be  a nonzero column $v$ of $A$.
        Choose a vector $w$ so that $(w^t\ 1)$ is a multiple of a row of the matrix $(A\ v)$.
    We have
    $$
    \rank \begin{pmatrix} A & v\\ w^t &1 \end{pmatrix}=\rank\begin{pmatrix}A & v\end{pmatrix}=\rank(A)=\max\{1,\rank(A)\}.
    $$
    \end{proof}

   Suppose that $T\in \F^{n_1}\otimes \cdots \otimes \F^{n_d}$.  
   Inductively, we will define partially filled matrices $A(T)$ and $B(T)$.
   If $d=2$ then $T$ can be viewed as a $n_1\times n_2$ matrix and we define $A(T)=\pad(T)$ and $B(T)=T$.

    We can write
    $$
    T=\sum_{i=1}^{n_d} S_i\otimes e_i.
    $$
    where $S_i\in \F^{n_1}\otimes \cdots \otimes \F^{n_{d-1}}$.
    We choose a linear isomorphism
    $$
    \phi:\F^{n_1}\otimes \cdots\otimes \F^{n_{d-1}}\to \F^{n_1n_2\cdots n_{d-1}}.
    $$
    where $\phi(S)$ will be viewed as a column vector.
    Define
    $$
    N(T)=\begin{pmatrix}
    \phi(S_1) & \cdots & \phi(S_{n_d})
    \end{pmatrix}
    $$
    and
    $$
    A(T)=
    \begin{pmatrix}
    \pad(N(T)) & & &\\
    & A(S_1) & & \\
    & & \ddots & \\
    & & & A(S_{n_d})\end{pmatrix}
    %,\quad
   %B(T)=
    %\begin{pmatrix}
   % N(T) & & &\\
   % & A(S_1) & & \\
   % & & \ddots & \\
   % & & & A(S_{n_d})\end{pmatrix}, 
    $$
    \begin{lemma}
    We have the inequality
    $$
    \minrank A(T)\geq \max\{1,\rank T\}+K(n_1,\dots,n_d),
    %\quad 
    %\minrank B(T)\geq \rank(T)+C,
    $$
    where 
    $$K(n_1,\dots,n_d)=n_3n_4\cdots n_d+n_4n_5\cdots n_d+\cdots+n_d.
    $$
    If $T$ has rank 0 or 1 then we have equality.
    \end{lemma}

 %   \begin{proposition}
   %   We have inequality
   % $$
   % \max\{\rank(T),1\}\leq \sum_{i=1}^{n_d} \max\{\rank(S_i)-1,0\}+\max\{\rank(N(T)),1\}.
    % $$
    % $$
   % \rank(T)\leq \sum_{i=1}^{n_d} \max\{\rank(S_i)-1,0\}+\rank(N(T)).
   %  $$

  %  with equality if $T$ has rank $0$ or $1$.
 %   \end{proposition}
    \begin{proof}
    We prove the inequality by induction on $d$. If $d=2$ then
    $A(T)=\pad(T)$ and $K(n_1,n_2)=0$. So we have
    $$\minrank A(T)=\minrank \pad(T)=\max\{\rank T ,1\}=\max\{\rank T,1\}+K(n_1,n_2).
    $$
    
    Suppose that $d\geq 3$. We prove the  inequality by induction on $n_d$.
    If $n_d=1$ then we have $T=S_1\otimes e_1$,  $\rank \pad(N(T))=1$, and 
    $$K(n_1,\dots,n_d)=K(n_1,\dots,n_{d-1},1)=K(n_1,\dots,n_{d-1})+1.
    $$
      It follows that
    \begin{multline*}
    \minrank A(T)=\minrank \pad(N(T))+\minrank A(S_1)\geq\\ \geq 1+\max\{1,\rank S_1\}
    +K(n_1,\dots,n_{d-1})=\max\{1,\rank T\}+K(n_1,\dots,n_d).
    \end{multline*}
    If $T$ has rank $\leq 1$, then $\rank S_1\leq 1$ and we have equality.
    
    Suppose that $n_d>1$.
    Let $T'=\sum_{i=1}^{n_d-1} S_i\otimes e_i$.
      %If $T=0$, then we have
    %\begin{multline*}
    %\minrank A(T)=1+\minrank A(S_1)+\dots+\minrank A(S_{n_d})=\\=
    %1+n_d(1+K(n_1,\dots,n_{d-1}))=1+K(n_1,\dots,n_d)=\max\{1,\rank(T)\}+K(n_1,\dots,n_d).
    %\end{multline*}
    %Define
    %$$
    %T'=\sum_{i=1}^{n_d-1} S_i\otimes e_i.
    %$$
    %If $S_{n_d}=0$, then
    %\begin{multline*}
    %\minrank A(T)=\minrank A(T')+\minrank A(S_{n_d})=\\=\max\{\rank(T'),1\}+K(n_1,\dots,n_{d-1},n_d-1)+
    %K(n_1,\dots,n_{d-1})+1=\\=\max\{\rank(T),1\}+K(n_1,\dots,n_d).
     %\end{multline*}
    
    First, assume that $S_1,\dots,S_{n_d}$ are linearly independent. 
    Then $T$ and $T'$ are nonzero. We have
    \begin{multline*}
    \minrank A(T)=\minrank \pad(N(T))+\sum_{i=1}^{n_d} \minrank A(S_i)=
    \rank N(T) +\sum_{i=1}^{n_d}\minrank A(S_i)=\\=
    \rank N(T')+1+\sum_{i=1}^{n_d}\minrank A(S_i)=
    \minrank A(T')+1+\minrank A(S_{n_d})\geq\\\geq
    \rank T'+K(n_1,\dots,n_{d-1},n_d-1)+1+\rank S_{n_d}+K(n_1,\dots,n_{d-1})=\\=
    \rank T'+\rank S_{n_d}+K(n_1,\dots,n_d)\geq \rank T+K(n_1,\dots,n_d).
    \end{multline*}
    
    If $S_1,\dots,S_{n_d}$ are linearly dependent, we may assume without
    loss of generality that $S_{n_d}$ lies in the span of $S_1,\dots,S_{n_d-1}$.
    Then we have
   $\rank T =\rank T'$, and 
   $$\minrank \pad(N(T))=\minrank \pad(N(T')).$$
   It follows that
   \begin{multline*}
   \minrank A(T)=\minrank \pad(N(T))+\sum_{i=1}^{n_d}\minrank A(S_i)=\\=
   \minrank \pad(N(T'))+\sum_{i=1}^{n_d}\minrank A(S_i)=
   \minrank A(T')+\minrank A(S_{n_d})\geq\\ \geq  \max\{1,\rank T' \}+K(n_1,\dots,n_{d-1},n_d-1)+
   \max\{1,\rank S_{n_d}\}+K(n_1,\dots,n_{d-1})\geq\\ \geq
    \max\{1,\rank T\}+K(n_1,\dots,n_{d-1},n_d-1)+K(n_1,\dots,n_{d-1})+1=\\ =\max\{1,\rank T\}+K(n_1,\dots,n_d).
   \end{multline*}
   If $\rank T\leq 1$ then $\rank S_i\leq 1$ for all $i$, and we have
   \begin{multline*}
   \minrank A(T)=\minrank \pad(N(T))+\sum_{i=1}^{n_d}\rank A(S_i)=\\=
   1+n_d(1+K(n_1,\dots,n_{d-1}))=1+K(n_1,\dots,n_d)=\max\{\rank T,1\}+K(n_1,\dots,n_d).
   \end{multline*}
   \end{proof}
   Let us define
   $$
   B(T)=\begin{pmatrix}   N(T) & & &\\
   & A(S_1) & &\\
   && \ddots &\\
   &&&& A(S_{n_d})\end{pmatrix}.
   $$
   \begin{lemma}
   We have
   $$
   \minrank B(T)\geq \rank T+K(n_1,\dots,n_d)
   $$
   with equality if $rank(T)\leq 1$.
   \end{lemma}
   \begin{proof}
   If $T$ is nonzero, then we have $\rank N(T)=\minrank \pad(N(T))$, so
   $$
   \minrank B(T)=\minrank A(T)\geq \rank T+K(n_1,\dots,n_d).
   $$
   with equality if $\rank T=1$.
   
   If $T=0$, then $\rank N(T)=0$ and $\rank \pad(N(T))=1$. It follows that
   $$
   \minrank B(T)=\minrank A(T)-1\geq K(n_1,\dots,n_d)=\rank T+K(n_1,\dots,n_d).
   $$
   \end{proof}
 For tensors $U_1,\dots,U_k$ we define
 $$
 C(U_1,\dots,U_k)=\begin{pmatrix}
 B(U_1-U_2) & & &\\
 & \ddots & &\\
 & & B(U_{k-2}-U_{k-1}) & \\
 & & & B(U_{k-1}-U_{k}).
 \end{pmatrix}
 $$
 \begin{lemma}
If $\rank T\leq k$, then we have
 $$
 \rank T+kK(n_1,\dots,n_d)=\min_{U_1,\dots,U_{k-1}} \minrank C(T,U_1,\dots,U_{k-1}).
 $$
 \end{lemma}
 \begin{proof}
 For all $U_1,\dots,U_{k-1}$ we have
\begin{multline*} 
\rank C(T,U_1,\dots,U_{k-1})=\rank B(T-U_1)+\sum_{i=2}^{k} \rank B(U_i-U_{i-1})\geq \\\geq
\rank(T-U_1)+\sum_{i=1}^k \rank(U_i-U_{i-1})+kK(n_1,\dots,n_d)\geq
\rank T+kK(n_1,\dots,n_d).
\end{multline*}
We can choose $U_1,\dots,U_{k-1}$ such that $T-U_1,U_1-U_2,\dots,U_{k-1}-U_k$ are pure tensors.
In that case we have equality.
\end{proof}
 
 {\bf Acknowledgment.} The author would like to thank Bernard Mourrain for pointing out some references.

 \end{document}